\documentclass[a4paper,12pt]{article}
\usepackage{amsmath, amsfonts, amssymb, amsthm}
\usepackage[english]{babel}

\newcounter{num}[section]

\newenvironment{theorem}
{\refstepcounter{num}%
\bigskip\noindent\nopagebreak[4]{\bf Theorem~\arabic{section}.\arabic{num}. }\it}
\newenvironment{corollary}
{\refstepcounter{num}%
\bigskip\noindent\nopagebreak[4]{\bf Corollary~\arabic{section}.\arabic{num}. }\it}
\newenvironment{lemma}
{\refstepcounter{num}%
\bigskip\noindent\nopagebreak[4]{\bf Lemma~\arabic{section}.\arabic{num}. }\it}
\newenvironment{remark}
{\refstepcounter{num}%
\bigskip\noindent\nopagebreak[4]{\bf Remark~\arabic{section}.\arabic{num}. }}

\newenvironment{example}
{\refstepcounter{num}%
\bigskip\noindent\nopagebreak[4]{\bf Example~\arabic{section}.\arabic{num}. }}

\newcommand{\Ss}{{\mathbf{S}}}
\newcommand{\V}{{\mathrm{V}}}

\newcommand{\lb}{{\langle}}
\newcommand{\rb}{{\rangle}}

\newcommand{\pr}{{\prime}}
\renewcommand{\t}{{\tau}}
\newcommand{\s}{{\sigma}}
\newcommand{\eps}{{\varepsilon}}
\newcommand{\al}{{\alpha}}

\newcommand{\F}{{\mathcal{F}}}
\newcommand{\I}{{\mathcal{I}}}

\renewcommand{\c}{{\mathbf{c}}}
\renewcommand{\d}{{\mathbf{d}}}
\renewcommand{\a}{{\mathbf{a}}}
\renewcommand{\b}{{\mathbf{b}}}
\newcommand{\e}{{\mathbf{e}}}
\newcommand{\f}{{\mathbf{f}}}

\newcommand{\ucl}{{\mathrm{Ucl}}}
\newcommand{\Rad}{{\mathrm{Rad}}}
\newcommand{\Hom}{{\mathrm{Hom}}}
\newcommand{\Core}{{\mathrm{Core}}}

\begin{document}

\author{Artem N. Shevlyakov}
\title{Algebraic geometry over free semilattice}

\maketitle

\section{Introduction}

In the papers~\cite{uniTh_I,uniTh_II} it was developed universal algebraic geometry over algebraic structures. In such papers it was introduced the definition of coordinate algebra (the analog of coordinate ring in commutative algebra) and explained that the classification of algebraic sets can be reduced to the classification of coordinate algebras.  The methods of these papers were applied to semigroups~\cite{AG_over_N_I,AG_over_N_II,AG_over_N_III}, and semilattices~\cite{at_service}. 

In the current paper we shall consider equations over semilattices, and any equation is an equality
\[
x_{i_1}x_{i_2}\ldots x_{i_n}\a=x_{j_1}x_{j_2}\ldots x_{j_m}\b,
\] 
where $\a,\b$ are elements of a semilattice.

Following the results of~\cite{uniTh_II}, the simplest description of coordinate algebras over an semilattice $S$ can be obtained if $S$ is equationally Noetherian. Remind that a semilattice $S$ is equationally Noetherian if any system of equations is equivalent over $S$ to its finite subsystem. Notice that we always consider systems of equations depending on at most finite set of variables.

In~\cite{at_service} we proved the criterion for an arbitrary semilattice to be equationally Noetherian. Following this result, any infinite semilattice is not equationally Noetherian.

Thus, we have to weaken the Noetherian property: a semilattice $S$ has the Noetherian property for consistent systems (NPC) if any consistent system of equations is equivalent over $S$ to its finite subsystem.  
In~\cite{at_service} there defined a semilattice $S$ with NPC, whereas $S$ is not equationally Noetherian Equivalently, there exists an infinite inconsistent system $\Ss$ over $S$ such that any finite subsystem of $\Ss$ is consistent. This fact does not contradict Malcev Compactness theorem. Indeed, Compactness theorem states that an infinite set $\Sigma$ of first-order sentences is satisfiable iff any finite subset $\Sigma_0\subseteq\Sigma$ is satisfiable. Notice that $\Sigma$ and all $\Sigma_0$ may be realized in different models, while NPC should be checked for systems of equations over the same semilattice $S$. 

Actually there are infinite semilattices with NPC. More precisely, it will be proven that any semilattice embeddable into a free semilattice of infinite rank has such property (Theorem~\ref{th:F_has_Noeth_prop}).

By Theorem~\ref{th:F_has_Noeth_prop} the free semigroup $\F$ of an infinite rank has NPC. It allows us apply almost all methods of the papers~\cite{uniTh_I,uniTh_II} to establish algebraic geometry over $\F$. In the current paper we obtain the description of coordinate semilattices of irreducible algebraic sets over $\F$. More precisely, we classify the coordinate semilattices using two approaches. Firstly, we define a universal semilattice $S$ such that any coordinate semilattice of an irreducible algebraic set over $\F$ is embedded into $S$ (the second statement of Theorem~\ref{th:main_irred}). Secondly, we write the list of axioms which logically define the class of coordinate semilattices of irreducible algebraic sets over $\F$ (the third statement of Theorem~\ref{th:main_irred}).

Theorem~\ref{th:main_irred} also allows us to 
\begin{enumerate}
\item define an algorithm which decomposes any algebraic set into a union of irreducible algebraic sets (Section~\ref{sec:algorithm});
\item obtain the parameterization of any algebraic set using the free parameters (Theorem~\ref{th:parameterization_of_arbitrary_set});
\end{enumerate}
Moreover, in Theorem~\ref{th:consistency} we solve the consistency problem of finite systems over $\F$.

\section{Semilattices and equations}

A {\it semilattice} is a commutative idempotent semigroup. We denote the multiplication in a semilattice by $\cdot$. By $\F$ we shall denote the free semilattice with free generators $\{\a_i|i\in\I\}$. The cardinality of the set $\I$ is the~\textit{rank} of the free semilattice $\F$. 

Indeed, any element $\a\in\F$ has a unique normal form
\[
\a=\a_{i_1}\a_{i_2}\ldots\a_{i_n},
\] 
where $i_1<i_2<\ldots<i_n$.

One can define a partial order for elements $\b,\c\in\F$ by
\[
\b\leq\c\Leftrightarrow \b\c=\b.
\]
The relation $\b\leq\c$ means that all free generators $\a_i$ occurring in the normal form of $\c$ should occur in $\b$. For example, $\a_1\a_2\a_3\leq\a_1\a_3$.

By $\F^\ast$ we denote the free semigroup $\F$ with the maximal element $1$ adjoined (i.e. $x\cdot 1=1\cdot x=x$ for all $x\in\F^\ast$).

Two elements $\b,\c\in\F$ are \textit{co-prime} if 
\[
\b=\prod_{i\in\I_\b}\a_i,\;\c=\prod_{i\in\I_\c}\a_i,
\]
and $\I_\b\cap\I_\c=\emptyset$. Or equivalently, elements $\b,\c\in\F$ are co-prime if there does not exists $\d\in\F$ with $\b\leq\d$, $\c\leq\d$. For example, the elements $\a_1\a_3$ and $\a_2\a_4$ are co-prime.

The following statement is well-known in semilattice theory.

\begin{theorem}
\label{th:embedd}
Any finitely generated semilattice is finite, and moreover it is embedded into a free semilattice of a  finite rank.
\end{theorem}

\medskip

Let $X=\{x_1,x_2,\ldots,x_n\}$ be a finite set of variables, and $F(X)$ be the free semilattice generated by the set $X$.

An $\F$-\textit{term} $\t(X)$ is an element of the free product $\F\ast F(X)$, i.e. $\t$ is one of the following expressions: $w(X)\a$, $w(X)$, $\a$, where $w(X)\in F(X)$ (a product of $x_i$), and $\a\in\F$. If $\t(X)$ does not contain any $\a\in\F$, the term $\t(X)$ is \textit{coefficient-free}.

\begin{remark}
Further we shall denote coefficient-free terms by Latin letters. The Greek letters denote $\F$-terms which may contain a constant. 
\end{remark}

\medskip

The equality of two $\F$-terms $\t(X)=\s(X)$ is called an {\it $\F$-equation}. For example, the next expressions are $\F$-equations $x_1x_2\a_1\a_2=x_1x_2\a_1\a_3\a_4$, $x_1x_2\a_2=x_3x_4$.

Often we shall use inequalities as equations, since
\[
\t(X)\leq \s(X)\Leftrightarrow \t(X)\s(X)=\t(X).
\]

One can naturally define the {\it solution set} $\V_\F(\t(X)=\s(X))$ of an $\F$-equation $\t(X)=\s(X)$ in the free semilattice $\F$. An arbitrary set of $\F$-equations is called a {\it system of equations} ({\it system} for shortness). Remark that we always consider systems depending on at most finite set of variables. The solution set of a system is the intersection of solution sets of its equations. If a system $\Ss$ has no solutions, it is called \textit{inconsistent}.

A set $Y\subseteq\F^n$ is called \textit{algebraic} if there exists a system of $\F$-equations with the solution set $Y$. An algebraic set is \textit{irreducible} if it is not a proper finite union of algebraic sets. 

One can naturally define the notion of $L$-equation (algebraic set) for any semilattice $L$ not merely for free semilattice $\F$. Hence, we formulate the next definitions in the general case.

Two systems of $L$-equations are called \textit{equivalent} if they have the same solution set over a semilattice $L$. 

A semilattice $L$ has the \textit{Noetherian property for consistent systems (NPC)} if any consistent system of $L$-equations is equivalent over $L$ to its finite subsystem.

A system of $L$-equations $\Ss$ is {\it homogeneous}, if there exist coefficient-free terms $t(X),s(X)$ such that $\Ss=\{t(X)\b_i=s(X)\c_i|i\in\mathcal{I},\b_i\c_i\in L^\ast\}$. Further we denote such system $\Ss$ by $\Ss_{t,s}$.

As the set of all pairs of coefficient-free terms is finite, any system of $L$-equations $\Ss$ over a semilattice $L$ is equivalent to a finite union 
\begin{equation}
\bigcup_{t,s}\Ss_{t,s}(X),
\label{eq:system_decomposition}
\end{equation}
where $\Ss_{t,s}(X)$ is the homogeneous system defined by the coefficient-free terms $t(X),s(X)$.

\bigskip

Let $Y\subseteq\F^n$ be an algebraic set over $\F$. The $\F$-terms $\s(X),\t(X)$ are equivalent if they have the same values at any point $P\in Y$. The set of such equivalence classes form a semilattice $\Gamma(Y)$ which is called the \textit{coordinate semilattice} of $Y$ (see~\cite{uniTh_I} for more details). 

A coordinate semilattice which corresponds to an irreducible algebraic set is called \textit{irreducible}. A coordinate semilattice determines an algebraic set up to isomorphism (the isomorphism of algebraic sets was defined in~\cite{uniTh_I}). Thus, one can consider the main aim of algebraic geometry as the classification of coordinate semilattices. 

Any semilattice with a fixed subsemilattice isomorphic to $\F$ is called an {\it $\F$-semilattice}. More formally, any semilattice $S$ with a fixed embedding $\eps_S\colon\F\to S$ is an $\F$-\textit{semilattice}. As any distinct constants $\a,\b\in\F$ are not equivalent with respect to the defined above relation, $\F$ is embedded into $\Gamma(Y)$. Hence, any coordinate semilattice over $\F$ is an $\F$-semilattice.

Notice that the set of $\F$-terms in variables $X$ is generated by the sets $X,\F$, therefore any coordinate $\F$-semilattice is finitely generated (as $\F$-algebra).

Let $S,L$ be $\F$-semilattices. A semilattice homomorphism $\psi\colon S\to L$ is called an $\F$-{\it homomorphism} if $\psi\circ\eps_S=\eps_L$ (i.e. $\psi$ fixes $\F$). The set of all $\F$-homomorphisms between $S,L$ is denoted by $\Hom_\F(S,L)$. A homomorphism $\psi\in\Hom_\F(S,L)$ is a $\F$-{\it embedding} if $\psi(s_1)\neq\psi(s_2)$ for all distinct $s_1,s_2\in S$. 

An $\F$-semilattice $S$ is $\F$-\textit{discriminated} by $L$ if for any distinct $s_1,s_2,\ldots,s_n\in S$ there exists $\psi\in\Hom_\F(S,L)$ with $\psi(s_i)\neq\psi(s_j)$ ($i\neq j$).

Let $\varphi$ be a first-order sentence of the language $\{\cdot\}\cup\{\a|\a\in\F\}$. The constants $\{\a|\a\in\F\}$ allow us to use explicitly the elements of $\F$ in $\varphi$. If a formula $\varphi$ holds in an $\F$-semilattice $S$ it is denoted by $S\models\varphi$. A formula $\varphi$ is \textit{universal} if it is equivalent to a formula
\[
\forall x_1\forall x_2\ldots\forall x_n \varphi^\pr(x_1,x_2,\ldots,x_n),
\] 
where $\varphi^\pr$ is quantifier-free.

The {\it universal closure} $\ucl(\F)$ of $\F$ consists of all $\F$-semilattices $S$ such that $S\models\varphi$ for any universal $\varphi$ with $\F\models\varphi$.

The next theorems were proved in~\cite{uniTh_I} for any algebraic structure, but we formulate it for $\F$.

\begin{theorem}\textup{\cite{uniTh_I}}
\label{th:coord=dis}
A finitely generated $\F$-semilattice $S$ is an irreducible coordinate semilattice iff $S$ is $\F$-discriminated by $\F$. 
\end{theorem}

\begin{theorem}\textup{\cite{uniTh_I}}
\label{th:coord->ucl}
If a finitely generated $\F$-semilattice $S$ is an irreducible coordinate semilattice then $S\in\ucl(\F)$. 
\end{theorem}

\medskip

A set $M$ of $\F$-equations is \textit{congruent} if the next conditions hold:
\begin{enumerate}
\item $\s(X)=\s(X)\in M$ for any $\F$-term $\s(X)$;
\item if $\s(X)=\t(X)\in M$ then $\t(X)=\s(X)\in M$;
\item if $\s(X)=\t(X)\in M$, $\t(X)=\rho(X)$ then $\s(X)=\rho(X)\in M$;
\item if $\s(X)=\t(X)\in M$, $\kappa(X)=\rho(X)$ then $\s(X)\kappa(X)=\t(X)\rho(X)\in M$.
\end{enumerate} 

A minimal congruent set including a system $\Ss$ is called the \textit{congruent closure} and denoted by $[\Ss]$. 

Clearly, the congruent closure of a system $\Ss$ contains only ``trivial'' consequences from the equation of the system. The next statement describes the conditions when a semilattice of a given presentation is a coordinate semilattice.

\begin{theorem}
A finitely generated $\F$-semilattice $S$ with a presentation
\[
S=\lb x_1,x_2,\ldots,x_n,\F|R\rb
\]
is a coordinate semilattice of an algebraic set over $\F$ iff the congruent closure of the relations $[R]$ coincides with the radical of $R$ 
\[
\Rad_{\F}(R)=\{\s(X)=\t(X)|\s(P)=\t(P)\mbox{ for all }P\in\V_\F(R)\}.
\]  
\label{th:definition_of_coord_semilattice}
\end{theorem} 
\begin{proof}
The proof is directly follows from the definition of coordinate semilattice.
\end{proof}

\section{Noetherian property for consistent systems}

In this section we describe the solutions sets of systems over a free semilattice $\F$ and prove that $\F$ has NPC.

Let us consider equations over $\F$ in at most two variables. Obviously, there exist three main types of such equations.  
\begin{enumerate}
\item $x\a\c=\a\d$,
\item $x\a\c=x\a\d$,
\item $x\a\c=y\a\d$
\end{enumerate}
for $\a,\b,\c\in\F^\ast$ and co-prime $\c,\d$.

Let us solve the equations above.
\begin{enumerate}
\item 
\[
\V_\F(x\a\c=\a\d)=\begin{cases}
\emptyset \mbox{ if } \c\neq 1,\\
\{\a^\pr \d|\a^\pr\geq\a\} \mbox{ otherwise }
\end{cases}.
\]
\item
\[
\V_\F(x\a\c=x\a\d)=\{\a^\pr \c\d t|t\in\F^\ast,\; \a^\pr\geq \a\},
\]
\item
\[
\V_\F(x\a\c=y\a\d)=\{(\a^\pr \d t,\a^{\pr\pr}\c t)|t\in\F^\ast,\; \a^\pr,\a^{\pr\pr}\geq \a\}.
\] 

\end{enumerate} 

Now we describe the solutions set of homogeneous systems in one or two variables over $\F$.

\begin{enumerate}
\item Suppose $\Ss=\{x\e_i\c_i=\e_i\d_i|i\in\I\}$. As the solution set of any equation from $\Ss$ is either finite or empty, $\Ss$ is equivalent to its finite subsystem.

\item Let $\c,\d\in\F$ be co-prime and  $\Ss=\{x\e_i\c=y\e_i\d|i\in\I\}$ an infinite system. Let $\e$ be the supremum of the set $\{\e_i\}$ in $\F^\ast$. By the properties of $\F^\ast$, $\e$ is a supremum of a finite subset $\{\e_1,\e_2,\ldots,\e_n\}\subseteq\{\e_i|i\in \I\}$. It is easy to check that $\Ss$ is equivalent to its finite subsystem $\{x\e_i\c=y\e_i\d|1\leq i\leq n\}$.

\item The system $\Ss=\{x\e_i\c_i=x\e_i\d_i|i\in\I\}$ can be obtained from the next system by the substitution $y:=x$.

\item Let us consider the most general type of a homogeneous system in two variables $\Ss=\{x\e_i\c_i=y\e_i\d_i|i\in\I\}$.  If the set $\{\c_i,\d_i|i\in \I\}$ has no infimum, $\Ss$ is inconsistent. Otherwise, let $\f=\inf\{\c_i,\d_i\}$. As there exists at most finite number of different $\c_i$ and $\d_i$ more than $\f$, the system $\Ss$ can be decomposed into a finite union
\[
\Ss=\bigcup_{\c,\d\geq\f}\Ss_{\c,\d}=\bigcup_{\c,\d\geq \f}\{x\e_j\c=y\e_j\d|j\in \I_{\c,\d}\}.
\]

Above we proved that any system $\Ss_{\c,\d}$ is equivalent to its finite subsystem, hence there exist a finite $\Ss^\pr$ which is equivalent to $\Ss$.
\end{enumerate}

Thus, we obtain the next result.

\begin{lemma}
Any consistent homogeneous system in at most two variables is equivalent to its finite subsystem.
\end{lemma}

\medskip

As any homogeneous system $\Ss_{t,s}$ can be reduced by the substitution $\mathbf{x}=t(X)$, $\mathbf{y}=s(X)$ to a homogeneous system in at most two variables, we come to

\begin{lemma}
Any consistent homogeneous system over $\F$ is equivalent to its finite subsystem.
\end{lemma}

\medskip

Since any system over $\F$ is a finite union of homogeneous ones (formula~\ref{eq:system_decomposition}), we have

\begin{theorem}
\label{th:F_has_Noeth_prop}
Any consistent system over $\F$ is equivalent to its finite subsystem. In other words, free semilattice $\F$ has NPC.
\end{theorem}

\begin{corollary}
Any semilattice $L$ embeddable into a free semilattice NPC.
\end{corollary}
\begin{proof}
Suppose a system of $L$-equations $\Ss$ is consistent over $L$. As $L\subseteq\F$, $\Ss$ is consistent over $\F$. By above, $\Ss$ is equivalent over $\F$ to its finite subsystem $\Ss^\pr$. By the inclusion $L\subseteq\F$, the systems $\Ss,\Ss^\pr$ have the same solution sets over $L$.
\end{proof}

\section{Irreducible coordinate semilattice over $\F$}

The aim of this section is to prove the next theorem.

\begin{theorem}
\label{th:main_irred}
Let $\F=\{\a_i|i\in\I\}$ be the free semilattice and $S$ a finitely generated $\F$-semilattice with $\Hom_\F(S,\F)\neq\emptyset$. The next conditions are equivalent:
\begin{enumerate}
\item $S$ is a coordinate semilattice over $\F$ of a nonempty irreducible algebraic set;
\item $S$ is $\F$-embeddable into the free product $\F\ast F(T)$, where $F(T)$ is a free semilattice generated by a finite set $T=\{t_1,t_2,\ldots,t_n\}$; 
\item $S\models\Sigma$, where $\Sigma=\{\varphi_i,\psi_{i}|i\in\I\}$,
where
\begin{equation}
\label{eq:varphi_a_i}
\varphi_i\colon\; \forall x,y \left( x\a_i=y\a_i\leftrightarrow (x\a_i=y \vee x=y\a_i \vee x=y) \right)
\end{equation}
\begin{equation}
\label{eq:psi_i}
\psi_i\colon\; \forall x,y \left( xy\leq \a_i\to (x\leq \a_i\vee y\leq \a_i)\right)
\end{equation}
for any free generator $\a_i$ of $\F$.
\end{enumerate}
\end{theorem}

Let us divide the proof into three subsections.

\subsection{$(1)\Rightarrow(3)$}

\begin{lemma}
\label{l:formulas_hold_in_F}
All formulas from $\Sigma$ hold in $\F$.
\end{lemma}
\begin{proof}
Firstly, we check $\F\models\varphi_i$. Take $x,y$ and solve the equation $x\a_i=y\a_i$. Clearly, the solution set equals to the union of three components:
\[
\V_\F(x\a_i=y\a_i)=\{(t,\a_it)|t\in\F\}\cup\{(\a_it,t)|t\in\F\}\cup\{(t,t)|t\in\F\}
\]

As for any point of the first (resp. second, third) component it holds $y=\a_ix$ (resp. $x=\a_iy$, $x=y$), the disjunction $x\a_i=y \vee x=y\a_i \vee x=y$ holds if assume $x\a_i=y\a_i$. Thus, $\F\models\varphi_i$.

Secondly, we prove  $\F\models\psi_i$. Consider the equation $xy\leq\a_i$ or equivalently $xy\a_i=xy$. As $\a_i$ occurs in the left part, the right part should contain $\a_i$. Hence, $x$ or $y$ contain $\a_i$. In other words, $x\leq\a_i$ or $y\leq\a_i$, and we obtain $\F\models\psi_i$.
\end{proof}

Thus all formulas from $\Sigma$ hold in $\F$. By Theorem~\ref{th:coord->ucl}, $S\in\ucl(\F)$, and $S\models\Sigma$.

\subsection{$(2)\Rightarrow(1)$}

It is sufficient to prove that $S=\F\ast F(T)$ is $\F$-discriminated by $\F$. Let $s_1,s_2,\ldots,s_n$ be pairwise distinct elements of $S$. Denote by $\mathcal{A}$ the set of all free generators $\a_i\in\F$ which occur in the words $\{s_i|1\leq i\leq n\}$. Without loss of generality one can put $\mathcal{A}=\{\a_1,\a_2,\ldots,\a_m\}$.

Consider a map $\psi\colon S\to \F$:
\[
\psi(t_i)=\a_{m+i}, \mbox{ for any }1\leq i\leq n \mbox{ and } \psi(\a)=\a\mbox{ for all }\a\in\F.
\]
It is easy to check that $\psi$ is an $\F$-homomorphism.

Let us prove that $\psi(s_i)\neq\psi(s_j)$ for any $i\neq j$.

Suppose $s_i=\a t$, $s_j=\a^\pr t^\pr$, where $\a,\a^\pr\in\F$, and $t,t^\pr$ are the words in $t_1,t_2,\ldots,t_n$.

\begin{enumerate}
\item Suppose $t\neq t^\pr$. Hence there exists $t_k$ which occurs in $t$ and does not in $t^\pr$, i.e $t_k\geq t$, $t_k\ngeq t^\pr$ (similarly, one can consider $t_k$ with $t_k\ngeq t$, $t_k\geq t^\pr$). Therefore, the image $\psi(t)$ contains $\a_{m+k}$, but $\psi(t^\pr)$ does not. By the definition of $\psi$, $\a_{m+k}$ does not occur in the words $\a,\a^\pr$. Thus, $\psi(s_i)\leq\a_{m+k}$, $\psi(s_j)\nleq\a_{m+k}$ and finally $\psi(s_i)\neq\psi(s_j)$.

\item Let $t=t^\pr$ and $\a\neq\a^\pr$. Therefore there exists $\a_k$ with $\a_k\geq\a$, $\a_k\ngeq \a^\pr$ (similarly, one can consider $\a_k$ with $\a_k\ngeq \a$, $\a_k\geq \a^\pr$). It means that the word $\a$ contains the generator $\a_k$, but $\a^\pr$ does not. By the definition of $\psi$, $\a_{k}$ does not occur in the words $\psi(t),\psi(t^\pr)$. Thus, $\psi(s_i)\leq\a_{k}$, $\psi(s_j)\nleq\a_{k}$ and finally $\psi(s_i)\neq\psi(s_j)$.
\end{enumerate}

Thus, we obtain $\psi(s_i)\neq\psi(s_j)$ for all $i\neq j$, hence $\F$ $\F$-discriminates $S$ and by Theorem~\ref{th:coord=dis} $S$ is a coordinate semilattice of an irreducible algebraic set over $\F$.

\subsection{$(3)\Rightarrow(2)$}

\begin{lemma}
The following formulas hold in $S$ for any $\a\in\F$:
\begin{equation}
\label{eq:varphi_a}
\varphi_\a\colon\; \forall x,y \left( x\a=y\a\leftrightarrow 
\bigvee_{{\a^\pr,\a^{\pr\pr}\geq\a }}x\a^\pr=y\a^{\pr\pr} \right)\;\mbox{\textup{for any co-prime} } \a^\pr,\a^{\pr\pr}\in\F^\ast,
\end{equation}

\end{lemma}
\begin{proof}
We prove the lemma by the induction on the length of the element $\a$. If $|\a|=1$ it is a free generator, and $\varphi_\a$ coincides with $\varphi_{i}$. Assume that the lemma holds for any $\a$ with $|\a|<n$. 

Let us prove the lemma for $\a$ of the length $n$. Let $\a=\bar{\a}\a_i$, where $|\bar{\a}|=n-1$. 

We have ($\a^\pr,\a^{\pr\pr}\in\F^\ast$ are co-prime below):
\begin{multline*}
x\a=y\a\leftrightarrow (x\a_i)\bar{\a}=(y\a_i)\bar{\a}\leftrightarrow 
\bigvee_{{\a^\pr,\a^{\pr\pr}\geq \bar{\a}}}(x\a_i)\a^\pr=(y\a_i)\a^{\pr\pr} \leftrightarrow \\
\bigvee_{{\a^\pr,\a^{\pr\pr}\geq \bar{\a}}}(x\a^\pr)\a_i=(y\a^\pr\pr)\a_i \leftrightarrow
\bigvee_{{\a^\pr,\a^{\pr\pr}\geq \bar{\a}}}x\a^\pr\a_i=y\a^{\pr\pr} \vee
\bigvee_{{\a^\pr,\a^{\pr\pr}\geq \bar{\a}}}x\a^\pr=y\a^{\pr\pr}\a_i \vee \\
\bigvee_{{\a^\pr,\a^{\pr\pr}\geq \bar{\a}}}x\a^\pr=y\a^{\pr\pr} \leftrightarrow
\bigvee_{{\a^\pr,\a^{\pr\pr}\geq \bar{\a}\a_i}}x\a^\pr=y\a^{\pr\pr},
\end{multline*}
and we obtain $\varphi_\a$.
\end{proof}

\begin{lemma}
The following formulas hold in $S$ for any co-prime $\a,\b\in\F$:
\begin{equation}
\label{eq:varphi_ab}
\varphi_{\a\b}\colon (x\a=y\b\to (x\leq\b)\cdot(y\leq\b))
\end{equation}
\end{lemma}
\begin{proof}
From the semilattice theory it holds $x\a\leq\b$. By formula $\psi_i$, any letter $\a_i$ in $\b$ should occur in $x\a$. As $\a$ and $\b$ are co-prime, $\a_i$ occurs in $x$. Thus, $x\leq \b$. Analogically, $y\leq\a$.
\end{proof}

Define a map $h\colon S\to\F^\ast$ such that the image $h(s)$ is the minimal element $\a\in\F^\ast$ with $s\leq\a$. Let us prove that the map $h$ is well-defined. If assume the existence of an infinite chain of elements $\b_1>\b_2>\ldots>\b_n>\ldots$ with $\b_i>s$, we obtain the emptiness of the set $\Hom_\F(S,\F)$ that contradicts with the condition.

\begin{lemma}
The map $h$ is an $\F$-homomorphism.
\end{lemma}
\begin{proof}
Obviously, $h(\a)=\a$ for all $\a\in\F$. 

Let us prove $h(s_1s_2)=h(s_1)h(s_2)$. Denote $h(s_i)=\b_i$.

As $s_i\leq \b_i$, we obtain $s_1s_2\leq\b_1\b_2$. Assume there exists $s_1s_2\leq\c<\b_1\b_2$.

As $\c<\b_1\b_2$ there exists a free generator $\a_{k}\geq\c$ that does not occur in $\b_1,\b_2$. By formula $\psi_i$ we have $s_1\leq\a_k$ or $s_2\leq\a_k$.

Suppose $s_1\leq\a_k$ (similarly, one can assume $s_2\leq\a_k$). By the choice of $\b_1$ we have $\b_1\leq\a_k$ and it follows that $\a_k$ occur in $\b_1$ that contradicts with the choice of the generator $\a_k$. 
\end{proof}

Let us define the equivalence relation over the semilattice $S$ by
\[
s_1\sim s_2 \Leftrightarrow \mbox{ there exists } \a,\b\in\F^\ast\mbox{ with } s_1\a=s_2\b.
\] 

By definition, for all $\a,\b\in\F$ we have $\a\sim \b$. Denote by $[s]$ the equivalence class which contains the element $s\in S$, and let $[\a]$ be the equivalence class which collects all elements of $\F$.

Suppose $s_1\sim s_2$, $t_1\sim t_2$, hence there exists $\a,\a^\pr,\b,\b^\pr\in\F^\ast$ with $s_1\a=s_2\a^\pr$, $t_1\b=t_2\b^\pr$.  Let us multiply the both equalities and obtain
\[
s_1t_1\a\b=s_2t_2\a^\pr\b^\pr,
\]
so $[s_1t_1]=[s_2t_2]$ and $\sim$ is the congruence.

Denote by $S^\pr$ the factor semilattice $S/\sim$ and $g\colon S\to S^\pr$ be the quotient map $g(S)=[s]$. 

For any $s\in S$ we have $[s][\a]=[\a][s]=[s]$. Hence $[\a]\in S^\pr$ is the identity element.

Let $f(s)=g(s)h(s)$ be a map between $S$ and $S^\pr\ast\F$. As $g(\F)=1$, and $h$ is an $\F$-homomorphism, hence $f(\a)=\a$ for any $\a\in\F$. Thus, $f\in\Hom_\F(S,S^\pr\ast\F)$

Let us prove that $f$ is an embedding. Assume the converse: there exists distinct $s,t$ with $f(s)=f(t)$. By the definitions of the homomorphisms $g,h$, the equality $f(s)=f(t)$ implies $g(s)=g(t)$ and $h(s)=h(t)$ in the semilattice $S^\pr\ast\F$. 

As $g(s)=g(t)$, there exists $\a,\b\in\F$ such that 
\begin{equation}
s\a=t\b.
\label{eq:s1}
\end{equation}

If $\a,\b$ are not co-prime, by formula~(\ref{eq:varphi_a}), in $S$ it holds 
\[
s\a^\pr=t\b^\pr,\; \a^\pr\geq\a,\; \b^\pr\geq\b
\] 
for some co-prime $\a^\pr,\b^\pr$.

Thus, one can initially assume that $\a,\b$ in the equality~(\ref{eq:s1}) are co-prime.

For co-prime $\a,\b$ one can apply formula $\varphi_{\a\b}$~(\ref{eq:varphi_ab}) and obtain $s\leq\b$, $t\leq\a$.

Denote $\c=h(s)=h(t)$. By the definition of $h$, $s\leq\c\leq\b$, $t\leq\c\leq\a$. Hence, $s\leq\a$, $t\leq\b$, and the equality~(\ref{eq:s1}) becomes $s=t$ that gives the contradiction.

We obtained that $S$ is $\F$-embedded into $S^\pr\ast\F$. 
As $S^\pr$ is a finitely generated semilattice, by Theorem~\ref{th:embedd} it is embedded into $F(t_1,t_2,\ldots,t_n)$ for an appropriate natural $n$ and we obtain an $\F$-embedding of $S$ into $F(t_1,t_2,\ldots,t_n)\ast\F$.

\section{Parameterization of systems over $\F$} 

%Firstly, we consider coefficient-free equations.

%\begin{lemma}
%\label{l:param_of_coeff_free_systems}
%Let $\Ss(x_1,x_2,\ldots,x_n)$ be a coefficient-free system of equations over $\F$. Then there exists coefficient-free terms $w_1(T),w_2(T),\ldots,w_n(T)$ in variables $T=\{t_1,t_2,\ldots,t_m\}$ such that
%\begin{equation}
%\V_\F(\Ss)=\{(w_1(T),w_2(T),\ldots,w_n(T))|t_i\in\F\}
%\label{eq:decompos_of_coeff_free_syst_with_free_gener}
%\end{equation}
%\end{lemma}
%\begin{proof}
%Let $S$ be the coordinate semilattice of the set $\V_\F(\Ss)$. As $S$ is equivalence classes of coefficient-free terms, $S$ does not contain the subalgebra $\F$ and moreover $S$ is generated by the elements $x_1,x_2,\ldots,x_n$.
%
%By Theorem~\ref{th:embedd}, $S$ is embedded into $F(T)$. Hence, any variable $x_i$ is a coefficient-free term $w_i(T)$ in variables $T=\{t_1,t_2,\ldots,t_m\}$. Thus, we obtain the formula~(\ref{eq:decompos_of_coeff_free_syst_with_free_gener}).
%\end{proof}

\begin{theorem}
A set $Y\subseteq\F^n$ is irreducible iff there exists coefficient-free terms $w_1(T),w_2(T),\ldots,w_n(T)$ in variables $T=\{t_1,t_2,\ldots,t_m\}$ and elements $\b_1,\b_2,\ldots,\b_n\in\F^\ast$ such that
\begin{equation}
\V_\F(\Ss)=\{(w_1(T)\b_1,w_2(T)\b_2,\ldots,w_n(T)\b_n)|t_i\in\F\}
\label{eq:decompos_of_irred_syst_with_free_gener}
\end{equation}
\end{theorem} 
\begin{proof}
Let $S$ be the coordinate $\F$-semilattice of the set $\V_\F(\Ss)$. By Theorem~\ref{th:main_irred}, $Y$ is irreducible iff any element $x_i\in S$ is represented by an element $w_i(T)\b_i\in\F\ast F(T)$ and we come to the formula~(\ref{eq:decompos_of_irred_syst_with_free_gener}).
\end{proof}

Let $\Ss$ be a system with no equations of the form $\s(X)=\a$. Denote by $\mathrm{cf}(\Ss)$ the system of coefficient-free equations such that $\mathrm{cf}(\Ss)$ is obtained from $\Ss$ by deleting all constants. For example, if $\Ss=\{x_1\a_1\a_3=x_1x_2\a_2\}$ then $\mathrm{cf}(\Ss)=\{x_1=x_1x_2\}$.

By Theorem~\ref{th:main_irred}, any system $\Ss$ with irreducible solution set does not contain equations $\s(X)=\a$, hence $\mathrm{cf}(\Ss)$ is defined for any system with the irreducible solution set. 

\begin{corollary}
Let $\Ss_1$, $\Ss_2$ be two systems of $\F$-equations with irreducible solution sets, and moreover $\mathrm{cf}(\Ss_1)=\mathrm{cf}(\Ss_2)$. Then $\Ss_1,\Ss_2$ have the same coefficient-free terms $w_i(T)$ in the formula~(\ref{eq:decompos_of_irred_syst_with_free_gener}).
\label{cor:same_cf(S)}
\end{corollary}
\begin{proof}
By the proof of Theorem~\ref{th:main_irred}, the terms $w_i$ define the embedding of a coordinate semigroup of the set $\mathrm{cf}(\Ss)$ into $F(T)$. The equality $\mathrm{cf}(\Ss_1)=\mathrm{cf}(\Ss_2)$ provides equal terms $w_i$ for the both systems.
\end{proof}

\begin{example}
Let us obtain a parameterization of the solution set of the system $\Ss=\{xy\a_1=z\a_2,y\leq\a_2\}$.

Firstly, we find the coefficient-free terms $w_i(T)$. Let us erase all constants from the equations of $\Ss$ (the inequality  $y\leq\a_2$ becomes a trivial equality $y=y$) and obtain $\mathrm{cf}(\Ss)=\{xy=z\}$. The coordinate semilattice of the set $\V_\F(\mathrm{cf}(\Ss))$ is presented as $S=\lb x,y,z|xy=z\rb$. The semilattice $S$ is embedded into $F(t_1,t_2)$ by $x\mapsto t_1$, $y\mapsto t_2$, $z\mapsto t_1t_2$. Thus, the elements $t_1,t_2,t_1t_2\in F(t_1,t_2)$ determine the terms $w_i(T)$ in the presentation of the set $Y$.

Now we obtain the estimation of variables by constants. By the properties of $\F$, the equality $xy\a_1=z\a_2$ implies $z\leq\a_1$, $xy\leq\a_2$. As $y\leq\a_2\in\Ss$, we have $xy\leq\a_2$ for any $x\in\F$, and there is no constraint for $x$. Thus, we obtain the correspondence
\[
x=t_1,\; y=t_2\a_2,\; z=t_1t_2\a_1,
\]   
and the solution set is 
\[
Y=\V_\F(\Ss)=\{(t_1,t_2\a_2,t_1t_2\a_1)|t_1,t_2\in\F\}.
\] 
\end{example}

\section{The consistency of finite systems over $\F$}
\label{sec:consistency}

In this section we establish a procedure which checks $\V_\F(\Ss)=\emptyset$ for a given finite system $\Ss$. 

\subsection{PROCEDURE I}

INPUT: a finite system of $\F$-equations $\Ss$ in variables $X=\{x_1,x_2,\ldots,x_n\}$.

\noindent OUTPUT: a subsystem $Sys\subseteq\Ss$, a set of variables $C$

\noindent STEP 0. Put $Sys:=\emptyset$, $C:=\emptyset$.

\noindent STEP $i$ ($i\geq 1$). Let $\Ss^\pr\subseteq\Ss$ be a set of all equations $\s(X)=\t(X)\in\Ss$ such that $\t(x)$ is either a constant or it depends only on the variables from the set $C$. Let $X_i=\{x_{i_1},x_{i_2},\ldots,x_{i_n}\}$ be all variables which occur in $\Ss^\pr$. Put
\[
C:=C\cup X_i,\; Sys:=Sys\cup\Ss^\pr
\]

If $\Ss^\pr=\emptyset$, terminate the procedure. 

The subsystem $Sys$ defined in Procedure I is called the \textit{core} of a system $\Ss$ and denoted by $\Core(\Ss)$. The variables occurring in equations of $\Core(\Ss)$ are \textit{fixed}.

\begin{example}
Let $\Ss$ be the next system
\[
\Ss=\{x_1x_2\a_2=\a_2,x_3\a_2=x_1\a_1,x_4\a_1=x_2x_3\a_2,x_5x_1\a_1\a_4=x_5x_4\a_2\}.
\] 

If we apply to $\Ss$ Procedure I, we consequently obtain the next sets of variables $X_1=\{x_1,x_2\}$, $X_2=\{x_3\}$, $X_3=\{x_4\}$ and the core is
\begin{equation}
\Core(\Ss)=\{x_1x_2\a_2=\a_2,x_3\a_2=x_1\a_1,x_4\a_1=x_2x_3\a_2\}.
\label{eq:core_from_example}
\end{equation}
The fixed variables of $\Ss$ are $x_1,x_2,x_3,x_4$.

\label{ex:system_for_consistency}
\end{example}

\begin{theorem}
Suppose for a finite system $\Ss(x_1,x_2,\ldots,x_m)$ Procedure I constructed a set $C$ and a core $\Core(\Ss)\subseteq\Ss$ in variables $x_1,x_2,\ldots,x_k$ ($k\leq n$) (w.l.o.g one can assume that the variables of $\Core(\Ss)$ have the first indexes). Let $\Ss_0=\Ss\setminus \Core(\Ss)$. Then:
\begin{enumerate} 
\item the solution set of $\Core(\Ss)$ is finite, and moreover it can be algorithmically found;
\item $\Ss$ is consistent iff so is $\Core(\Ss)$;
\item if $\V_\F(\Core(\Ss))=\{P_1,P_2,\ldots,P_m\}$, where 
\[
P_i=(p_{i1},p_{i2},\ldots,p_{ik})
\]
then
\begin{equation}
\V_\F(\Ss)=\bigcup_{i=1}^m\V_\F\left(\Ss_0\bigcup_{j=1}^k\{x_j=p_{ij}\}\right).
\label{eq:union_of_S_without_core}
\end{equation}

\end{enumerate}
\label{th:consistency}
\end{theorem}

Let us explain the statement of Theorem~\ref{th:consistency} by the system from Example~\ref{ex:system_for_consistency}

\begin{example}
The solution set of~(\ref{eq:core_from_example}) is finite:
\[
\V_\F(\Core(\Ss))=\{(\a_2,\a_2,\a_1,\a_2),(\a_2,\a_2,\a_1,\a_1\a_2),(\a_2,\a_2,\a_1\a_2,\a_2),(\a_2,\a_2,\a_1\a_2,\a_1\a_2)\}
.
\]
The consistency of $\Core(\Ss)$ admits a solution of the whole system $\Ss$. Clearly,
$(\a_2,\a_2,\a_1,\a_2,\a_1\a_2\a_4)\in\V_\F(\Ss)$ that adjusts with the second statement of Theorem~\ref{th:consistency}.

By formula~(\ref{eq:union_of_S_without_core}), the system $\Ss$ really equals to the union:
\begin{multline*}
\Ss=\begin{cases}
x_1=\a_2,\\
x_2=\a_2,\\
x_3=\a_1,\\
x_4=\a_2,\\
x_5x_1\a_1\a_4=x_5x_4\a_2,x_1
\end{cases}
\cup
\begin{cases}
x_1=\a_2,\\
x_2=\a_2,\\
x_3=\a_1,\\
x_4=\a_1\a_2,\\
x_5x_1\a_1\a_4=x_5x_4\a_2,x_1
\end{cases}
\cup
\begin{cases}
x_1=\a_2,\\
x_2=\a_2,\\
x_3=\a_1\a_2,\\
x_4=\a_2,\\
x_5x_1\a_1\a_4=x_5x_4\a_2,x_1
\end{cases}\\
\cup
\begin{cases}
x_1=\a_2,\\
x_2=\a_2,\\
x_3=\a_1\a_2,\\
x_4=\a_1\a_2,\\
x_5x_1\a_1\a_4=x_5x_4\a_2,x_1
\end{cases}.
\end{multline*}

\end{example}

Now we begin to prove Theorem~\ref{th:consistency}.

\begin{proof}
\begin{enumerate}
\item By the definition of Procedure I, the set of variables $X_c$ of $\Core(\Ss)$ is the disjoint union
\[
X_c=X_1\cup X_2\cup\ldots\cup X_l.
\]

Let $\Ss_i(X_i)\subseteq\Ss$ be the system which was added to $\Core(\Ss)$ at the $i$-th step of Procedure I. Let us prove that $|\V_\F(\Ss_1)|<\infty$.

By the choice of $\Ss_1$, it contains equation of the form $s(X_1)\a=\b$. 

By the properties of $\F$, any variable $x$ occurring in $s(X_1)$ should satisfy $x\geq\b$. As the set $\b\uparrow=\{x|x\geq\b\}\subseteq\F$ is finite, the solution set of $s(X_1)\a=\b$ is also finite. Thus, we obtain $|\V_\F(\Ss_1)|<\infty$. As we should seek a solution of $s(X_1)\a=\b$ in the finite set $\b\uparrow$, this problem is obviously algorithmically decidable.

The finiteness of the solution set of $\Ss_i$ for $1< i\leq l$ (and the algorithmic decidability) can be easily proven by the induction.

\item As $\Core(\Ss)\subseteq\Ss$, the equality $\Core(\Ss)=\emptyset$ implies $\V_\F(\Ss)=\emptyset$. 

Suppose now $P=(p_1,p_2,\ldots,p_k)\in\V_\F(\Core(\Ss))$. Put $\a$ be the product of all letters $\a_j$ which occur in $\V_\F(\Core(\Ss))$ and $\b$ be the product of all constants from all equations of $\Ss_0$ (for instance, in Example~\ref{ex:system_for_consistency} we have $\a=\a_1\a_2$, $\b=\a_1\a_2\a_4$). The definitions of $\a,\b$ are well-defined, as the sets $\V_\F(\Core(\Ss))$, $\Ss_0$ are finite. Further, we put $\c=\a\b$. 

Remind that the system $\Ss_0$ depends on the variables $x_{k+1},x_{k+2},\ldots,x_n$. Let us prove $Q=(\c,\c,\ldots,\c)\in\V_\F(\Ss_0)$. Take an arbitrary equation 
\[
s(x_{k+1},x_{k+2},\ldots,x_n)\d=t(x_{k+1},x_{k+2},\ldots,x_n)\e\in\Ss_0
\]
and obtain
\[
s(\c,\c,\ldots,\c)\d=t(\c,\c,\ldots,\c)\e\Leftrightarrow \c\d=\c\e.
\]
By the choice of $\c$, $\c\leq\d$, $\c\leq\e$, and we come to the true equality $\c\d=\c\e\Leftrightarrow\c=\c$.

Thus, $Q\in\V_\F(\Ss_0)$. Let us joint the points $P,Q$ by
\[
R=(p_1,p_2,\ldots,p_k,\c,\c,\ldots,\c).
\]   
As the systems $\Core(\Ss),\Ss_0$ have the disjoint sets of variables, the point $R$ is the solution of the whole system $\Ss$. Hence, $\Ss$ is consistent.

\item Straightforward. 

\end{enumerate}

\end{proof}

\section{Decompositions of algebraic sets}
\label{sec:algorithm}

We introduce two procedures that form the Decomposition algorithm defined below.

\subsection{PROCEDURE II}

INPUT: a finite system $\Ss$ which does not contain any equation of the type $\s(X)=\b$ ($\b\in\F$).

\noindent OUTPUT: a set of systems $D=\{\Ss_1,\Ss_2,\ldots,\Ss_m\}$ such that 
\[
\V_\F(\Ss)=\bigcup_{i=1}^m\V_\F(\Ss_i).
\]
\noindent STEP 0. Put $D:=\Ss$.

\noindent STEP 1. Find an equation $\s(X)\a=\t(X)\a$ such that
\begin{enumerate}
\item $\a\in\F$;
\item $\s(X)\a=\t(X)\a$ belongs to all system from $D$;
\item the constants of $\s(X)$ and $\t(X)$ are co-prime.
\end{enumerate}
If such equation does not exists, terminate the procedure. By formula $\varphi_{\a}$, the given equation is equivalent to the union 
\[
\s(X)\a=\t(X)\a\sim \bigcup_{\a^\pr,\a^{\pr\pr}\geq\a}\s(X)\a^\pr=\t(X)\a^{\pr\pr}
\] 
for all co-prime $\a^\pr,\a^{\pr\pr}$.

Replace any $\Ss^\pr\in D$ to the collection systems
\[
C(\Ss^\pr)=\{\Ss_{\a^\pr\a^{\pr\pr}}|\a^\pr,\a^{\pr\pr}\geq\a,\mbox{ and }\a^\pr,\a^{\pr\pr} \mbox{ are co-prime}\},
\]
\[
\Ss_{\a^\pr\a^{\pr\pr}}=(\Ss^\pr\setminus\{\s(X)\a=\t(X)\a\})\cup\{\s(X)\a^\pr=\t(X)\a^{\pr\pr}\}.
\]

\begin{example}
Let us explain the work of Procedure II at the next system:
\[
\Ss=\begin{cases}
xy\a_1=x\a_1\a_2,\\
xz\a_1\a_2\a_3=y\a_2\a_3
\end{cases}.
\]

Firstly, $D=\{\Ss\}$ and take $xy\a_1=x\a_1\a_2\in\Ss$ which is equivalent to the union
\[
xy=x\a_1\a_2 \cup xy\a_1=x\a_2 \cup xy=x\a_2.
\]
Thus, we have three systems instead of $\Ss$:
\[
\Ss_1=\{xy=x\a_1\a_2,xz\a_1\a_2\a_3=y\a_2\a_3\},
\]
\[
\Ss_2=\{xy\a_1=x\a_2,xz\a_1\a_2\a_3=y\a_2\a_3\},
\]
\[
\Ss_3=\{xy=x\a_2,xz\a_1\a_2\a_3=y\a_2\a_3\},
\]
and $D=\{\Ss_1,\Ss_2,\Ss_3\}$ after the first step of the procedure.

At the second step we take the equation $xz\a_1\a_2\a_3=y\a_2\a_3$. It is equivalent to the union:
\begin{multline}
xz\a_1=y\a_2\a_3\cup xz\a_1\a_2\a_3=y\cup xz\a_1=y\cup xz\a_1\a_2=y\a_3\cup xz\a_1\a_2=y\cup\\
xz\a_1\a_3=y\a_2\cup xz\a_1\a_3=y\cup xz\a_1=y\a_2\a_3\cup xz\a_1=y\a_2\cup xz\a_1=y\a_3.
\label{multlinchik}
\end{multline}

The final set $D$ consists of 30 systems $\Ss_{ij}$, $1\leq i\leq 3$, $1\leq j\leq 10$, where 
\[
\Ss_{ij}=(\Ss_i\setminus\{xz\a_1\a_2\a_3=y\a_2\a_3\})\cup\{\mbox{the $j$-th equation from~(\ref{multlinchik})}\}.
\] 

\end{example}

\subsection{PROCEDURE III}

INPUT: a finite system $\Ss$ of $N$ equations which does not contain neither $\s(X)=\b$ nor $\s(X)\b=\t(X)\b$ ($\a\in\F^\ast$, $\b\in\F$).

\noindent OUTPUT: a set of systems $D=\{\Ss_1,\Ss_2,\ldots,\Ss_m\}$ such that 
\[
\V_\F(\Ss)=\bigcup_{i=1}^m\V_\F(\Ss_i).
\]
\noindent STEP 0. Put $D:=\Ss$.

\noindent STEP $i$ ($1\leq i\leq N$). Take the $i$-th equation in $\Ss$: $s(X)\a=t(X)\b\in\Ss$, where $\a,\b\in\F^\ast$. By the condition, $\a,\b$ are co-prime (here we assume that $1\in\F^\pr$ and any $\c\in\F^\ast$ are co-prime). The semilattice theory gives $s(X)\leq\b$, $t(X)\leq\a$. Suppose 
\[
\a=\prod_{i\in\I_\a}\a_{i},\; \b=\prod_{i\in\I_\b}\a_{i},\; s(X)=\prod_{i\in\I_s}x_i,\; t(X)=\prod_{i\in\I_t}x_i.
\]

The inequalities $s(X)\leq\b$, $t(X)\leq\a$ implies
\[
t(X)\leq\a_i (i\in\I_\a),\; s(X)\leq\a_i (i\in\I_\b).
\]

By formula~(\ref{eq:psi_i}), we have
\begin{equation}
\begin{cases}
\bigcup_{j\in\I_t}(x_j\leq\a_i)\mbox{ for any }i\in\I_\a,\\
\bigcup_{j\in\I_s}(x_j\leq\a_i)\mbox{ for any }i\in\I_\b.
\end{cases}
\label{eq:trololo}
\end{equation}
Using the distributivity law for algebraic sets, one can rewrite~(\ref{eq:trololo}) as a union 
\[
\bigcup_{\al\in M} \Ss_\al,
\]
where $\Ss_\al$ is a system of equations of the type $x_i\leq\a_j$ ($i\in\I_s\cup\I_t$, $j\in\I_\a\cup\I_\b$).

Replace any $\Ss^\pr\in D$ by the collection of the systems 
\[
\{\Ss^\pr\cup\Ss_\al|\al\in M\}.
\] 

\begin{example}
\label{ex:for_procedure_III}
Let us demonstrate the work of Procedure III at the next example.

Suppose $\Ss=\{x_2\a_1=x_1x_3\a_2,x_2x_4\a_3=x_3\a_4\a_5\}$. At the first step of the procedure we deal with $x_2\a_1=x_1x_3\a_2$ and obtain $x_2\leq\a_2$, $x_1x_3\leq\a_1$. We have
\[
\begin{cases}
x_2\leq\a_2,\\
x_1\leq\a_1\cup x_3\leq\a_1
\end{cases}
\] 
or equivalently
\[
(x_2\leq\a_2)(x_1\leq\a_1)\cup(x_2\leq\a_2)(x_3\leq\a_1).
\]

Thus, after the first step we have
\begin{multline*}
D=\{\{x_2\a_1=x_1x_3\a_2,x_2x_4\a_3=x_3\a_4\a_5,x_2\leq\a_2,x_1\leq\a_1\},\\
\{x_2\a_1=x_1x_3\a_2,x_2x_4\a_3=x_3\a_4\a_5,x_2\leq\a_2,x_3\leq\a_1\}\}=\{\Ss_1,\Ss_2\}.
\end{multline*}

At the second step we take $x_2x_4\a_3=x_3\a_4\a_5$, hence $x_2x_4\leq\a_4$, $x_2x_4\leq\a_5$, $x_3\leq\a_3$ and
\[
\begin{cases}
x_2\leq\a_4\cup x_4\leq\a_4,\\
x_2\leq\a_5\cup x_4\leq\a_5,\\
x_3\leq\a_3
\end{cases}
\]
or equivalently
\begin{multline*}
(x_2\leq\a_4)(x_2\leq\a_5)(x_3\leq\a_3)\cup(x_2\leq\a_4)(x_4\leq\a_5)(x_3\leq\a_3)\cup
(x_4\leq\a_4)(x_2\leq\a_5)(x_3\leq\a_3)\cup\\
(x_4\leq\a_4)(x_4\leq\a_5)(x_3\leq\a_3).
\end{multline*}

Finally, the set $D$ after the second step becomes
\begin{multline*}
D=\{\Ss_1\cup\{x_2\leq\a_4,x_2\leq\a_5,x_3\leq\a_3\},\Ss_1\cup\{x_2\leq\a_4,x_4\leq\a_5,x_3\leq\a_3\},\\
\Ss_1\cup\{x_4\leq\a_4,x_2\leq\a_5,x_3\leq\a_3\},\Ss_1\cup\{x_4\leq\a_4,x_4\leq\a_5,x_3\leq\a_3\},\\
\Ss_2\cup\{x_2\leq\a_4,x_2\leq\a_5,x_3\leq\a_3\},\Ss_2\cup\{x_2\leq\a_4,x_4\leq\a_5,x_3\leq\a_3\},\\
\Ss_2\cup\{x_4\leq\a_4,x_2\leq\a_5,x_3\leq\a_3\},\Ss_2\cup\{x_4\leq\a_4,x_4\leq\a_5,x_3\leq\a_3\}
\}.
\end{multline*}

One can rewrite $D$ in a simpler form:
\begin{multline*}
D=\{\Ss\cup\{x_1\leq\a_1,x_2\leq\a_2\a_4\a_5,x_3\leq\a_3\},\Ss\cup\{x_1\leq\a_1,x_2\leq\a_2\a_4,x_3\leq\a_3,x_4\leq\a_5\},\\
\Ss\cup\{x_1\leq\a_1,x_2\leq\a_2\a_5,x_3\leq\a_3,x_4\leq\a_4\},\Ss\cup\{x_1\leq\a_1,x_2\leq\a_2,x_3\leq\a_3,x_4\leq\a_4\a_5\},\\
\Ss\cup\{x_2\leq\a_2\a_4\a_5,x_3\leq\a_1\a_3\},\Ss\cup\{x_2\leq\a_2\a_4,x_3\leq\a_1\a_3,x_4\leq\a_5\},\\
\Ss\cup\{x_2\leq\a_2\a_5,x_3\leq\a_1\a_3,x_4\leq\a_4\},\Ss\cup\{x_2\leq\a_2\a_5,x_3\leq\a_1\a_3,x_4\leq\a_4\a_5\}\}.
\end{multline*}

\end{example}

\begin{lemma}
Let $\bar{\Ss}$ be a finite system which does not contain neither $\s(X)=\b$ nor $\s(X)\b=\t(X)\b$ ($\a\in\F^\ast$, $\b\in\F$). Hence, Procedure III constructs a set $D$ which consists of systems with irreducible solution sets.
\label{l:about_Procedure_III}
\end{lemma}
\begin{proof}
Let $\Ss\in D$ and $X=\{x_1,x_2,\ldots,x_n\}$ be the variables which occur in $\Ss$. Define an $\F$-semilattice $S$ with generators $X$ and relations $\Ss$. 

By condition, $\Ss$ does not contain equations $\s(X)\a=\b$, hence, we have $\Core(\Ss)=\emptyset$. By Theorem~\ref{th:consistency}, $\Ss$ is consistent, or equivalently, the semilattice $S$ has a non-empty set of $\F$-homomorphisms $\Hom_\F(S,\F)$.

Let us check $S\models\Sigma$. Suppose there exists words $\s(X),\t(X)\in S$ with $\s(X)\a=\t(X)\a\in[\Ss]$. 
Without loss of generality, one can assume that $\a$ is maximal among all $\b\in\F$ which occur in equations $\s^\pr(X)\b=\t^\pr(X)\b\in[\Ss]$. 

By the definition of $[\Ss]$, there exist equations $\s_i(X)=\t_i(X)\in\Ss$, $1\leq i\leq m$, and trivial equality $\rho(X)=\rho(X)$ such that
\[
\s(X)\a=\s_1(X)\s_2(X)\ldots \s_m(X)\rho(X),
\]
\[
\t(X)\a=\t_1(X)\t_2(X)\ldots \t_m(X)\rho(X).
\]

If $\rho(X)=r(X)\b$, we have
\[
\s_1(X)\s_2(X)\ldots \s_m(X)r(X)=\s(X)\a^\pr,
\]
\[
\t_1(X)\t_2(X)\ldots \t_m(X)r(X)=\t(X)\a^\pr
\]
for some $\a^\pr$ with $\a^\pr\b=\a$. Hence, $\s(X)\a^\pr=\t(X)\a^\pr\in[\Ss]$ for some $a^\pr\geq\a$ that contradicts with the choice of $\a$.
Thus, $\rho(X)$ is coefficient-free.

Let $\s_i(X)=s_i(X)\b_i$, $\t_i(X)=t_i(X)\c_i$, hence 
\begin{equation}
\label{eq:qqqqq}
\prod_i \b_i=\prod_i\c_i=\a.
\end{equation}

By the definition of Procedure III, $[\Ss]$ contains $s_i(X)\leq\c_i$, $t_i(X)\leq\b_i$, and from~(\ref{eq:qqqqq}) it implies 
\[
\prod_i s_i(X)\leq\a,\;\prod_i t_i(X)\leq\a.
\]

Therefore, 
\[
\s(X)\a=\prod_{i}s_i(X)\a=\prod_{i}s_i(X),
\]
\[
\t(X)\a=\prod_{i}t_i(X)\a=\prod_{i}t_i(X),
\]
and the equality $\s(X)=\t(X)$ holds. 

Remember, in the origin of the proof we assume $\s(X)\a=\t(X)\a$, hence the implication $(\s(X)\a=\t(X)\a)\to(\s(X)=\t(X))$ holds for any $X$. 
Thus, $S\models\varphi_{\a}$.

\medskip

Let us check $S\models\psi_k$. Assume the converse: $\psi_k$ does not hold in $S$, i.e. there exists $\s(X)\leq\a_k\in[\Ss]$ and all variables of $\s(X)$ do not satisfy $x\leq\a_k\in[\Ss]$.  As $\s(X)\leq\a_k\Leftrightarrow \s(X)\a_k=\s(X)$ and $S\models \varphi_\a$, we may assume that $\s(X)$ is coefficient-free. Further we denote $\s(X)=s(X)$ and we have 
\[
s(X)\a_k=s(X)\in[\Ss].
\]

There exists equations $\s_i(X)=\t_i(X)\in\Ss$ and a trivial equality $\rho(X)=\rho(X)$ such that
\[
s(X)\a_k=\rho(X)\prod_i^k \s_i(X), \;s(X)=\rho(X)\prod_i^k \t_i(X).
\]
It follows that 
\begin{enumerate}
\item the terms $\rho(X)=r(X),\t_i(X)=t_i(X)$ are coefficient-free;
\item there exists $i_0$ such that $\s_{i_0}(X)=s_{i_0}(X)\a_k$.
\end{enumerate}

As $s_{i_0}(X)\a_k=t_{i_0}(X)\in\Ss$, by the definition of Procedure III there exists an equation $x\leq\a_k\in\Ss$ for some $x$ occurring in $t_{i_0}(X)$.

As $s(X)$ is the product of $t_i(X)$, $s(X)$ contains $x$. Thus, $s(X)\leq\a_k\in[\Ss]$ that contradicts with the assumption.

Finally, we proved $S\models\Sigma$. By Theorem~\ref{th:main_irred}, $S$ is a coordinate semilattice of an irreducible algebraic set over $\F$. Thus, the solution set of $\Ss$ is irreducible.
\end{proof}

\subsection{Decomposition algorithm}

INPUT: a finite system $\Ss$.

\noindent OUTPUT: a set of systems $D$ such that
\begin{enumerate}
\item $D=\emptyset$, if $\Ss$ is inconsistent;

\item $D=\{\Ss_1,\Ss_2,\ldots,\Ss_m\}$ and
\[
\V_\F(\Ss)=\bigcup_{\Ss_i\in D}\V_\F(\Ss_i),
\]
where all $\Ss_i$ are consistent and have irreducible solution sets. 
\end{enumerate}

\noindent STAGE 1: apply Procedure I from Section~\ref{sec:consistency} for the system $\Ss$. If a subsystem $Sys=\Core(\Ss)$ is inconsistent, put $D=\emptyset$ and terminate the algorithm.

\noindent STAGE 2: apply Procedure II for the system $\Ss^\pr=\Ss\setminus\Core(\Ss)$. The procedure gets a set of systems $D^\pr$.

\noindent STAGE 3: for any $\Ss^\pr_i\in D^\pr$ launch Procedure III which gives a set of systems $D_i^\pr$.

\noindent STAGE 4: Put $D=\bigcup_{i}D^\pr_i$ and terminate the algorithm. 

\noindent STAGE 5: Without loss of generality one can put that $\Core(\Ss)$ contains the variables $x_1,x_2,\ldots,x_m$. Solve $\Core(\Ss)$ and obtain the set of points
\[
P_j=(p_{1j},p_{2j},\ldots,p_{nj}),\; 1\leq j\leq m.
\]

Replace any system $\Ss^\pr\in D$ to the collection of systems
\[
\Ss_j^\pr=\Ss^\pr\cup\{x_1=p_{1j},x_2=p_{2j},\ldots,x_m=p_{jm}\}.
\]

\begin{theorem}
\label{th:coorect_decomp_alg}
The algorithm above check the inconsistency of a given system $\Ss$. If $Y=\V_\F(\Ss)\neq \emptyset$, it finds the decomposition of the solution set $Y=\V_\F(\Ss)$ into a union of irreducible algebraic sets $\{\V_\F(\Ss_i)|\Ss_i\in D\}$.
\end{theorem}
\begin{proof}
Following Theorem~\ref{th:consistency}, for the inconsistency of a system $\Ss$ it is sufficient to check $\Core(\Ss)=\emptyset$ that was made by the algorithm at the Stage 1. 

The application of Lemma~\ref{l:about_Procedure_III} concludes the proof. 
\end{proof}

Using the decomposition algorithm, one can apply the formula~(\ref{eq:decompos_of_irred_syst_with_free_gener}) for any algebraic set not merely for irreducible one. Precisely, the following theorem holds.

\begin{theorem}
\label{th:parameterization_of_arbitrary_set}
Let $Y\subseteq\F^n$ be an algebraic set which is a union of $k$ irreducible sets. Hence, there exist a set of variables $T=\{t_1,t_2,\ldots,t_m\}$, coefficient-free terms $\{w_i(T)|1\leq i\leq n\}$ and constants $\{\b_{ij}|1\leq i\leq n,1\leq j\leq k\}$, $\b_{ij}\in\F^\ast$ such that
\begin{equation}
Y=\bigcup_{j=1}^k\{(w_1(T)\b_{1j},w_2(T)\b_{2j},\ldots,w_n(T)\b_{nj})|t_i\in\F\}.
\label{eq:parametrization_of_arbitrary_algebraic_set}
\end{equation}
\end{theorem} 
\begin{proof}
Let $D=\{\Ss_1,\Ss_2,\ldots,\Ss_k\}$ be a set of system obtained by Decomposition algorithm.
It is easy to see that the system of $D$ differ from each other by equations of the form $x_i\leq\c_i$ or $x_i=\d_i$. Thus, the systems $\Ss_j$ have the common part $\Ss_0$ and any $\Ss_j$ can be written in the form
\begin{equation*}
\Ss_j=\Ss_0\bigcup_{i=1}^n{x_i\#\c_{ij}},
\label{eq:camel}
\end{equation*}
where $\#$ is either $\leq$ or $=$. 

By~(\ref{eq:camel}), we have $\mathrm{cf}(\Ss_1)=\mathrm{cf}(\Ss_2)=\ldots=\mathrm{cf}(\Ss_k)$. According Corollary~\ref{cor:same_cf(S)} the representations~(\ref{eq:decompos_of_irred_syst_with_free_gener}) of all $\V_\F(\Ss_j)$ have the same terms $w_1(T),w_2(T),\ldots,w_n(T)$ and we come to~(\ref{eq:parametrization_of_arbitrary_algebraic_set}).
\end{proof}

\begin{example}
\label{ex:last}
Let us explain the formula~(\ref{eq:parametrization_of_arbitrary_algebraic_set}) at the system $\Ss$ from Example~\ref{ex:for_procedure_III}.

Let $\Ss^\pr$ be the system obtained from $\Ss$ by deleting all constants and renaming the variables:
\[
\Ss^\pr=\{z_2=z_1z_3,z_2z_4=z_3\}.
\]

One can introduce free generators $T=\{t_1,t_2\}$ and prove that the solution set of $\Ss^\pr$ is
\[
\V_\F(\Ss^\pr)=\{(t_1,t_1t_2,t_1t_2,t_2)|t_i\in\F\}.
\]

Thus, the terms $w_i(T)$ from formula~(\ref{eq:parametrization_of_arbitrary_algebraic_set}) are $w_1(T)=t_1$, $w_2(T)=t_1t_2$, $w_3(T)=t_1t_2$, $w_4(T)=t_2$.

The constants $\b_{ij}$ ($1\leq i\leq 4$, $1\leq j\leq 8$) are obtained from the inequalities of $j$-th system from $D$. Let us write $\b_{ij}$ as matrix elements
\[
\begin{pmatrix}
\a_1&\a_1&\a_1&\a_1&1&1&1&1\\
\a_2\a_4\a_5&\a_2\a_4&\a_2\a_5&\a_2&\a_2\a_4\a_5&\a_2\a_4&\a_2\a_5&\a_2\a_5\\
\a_3&\a_3&\a_3&\a_3&\a_1\a_3&\a_1\a_3&\a_1\a_3&\a_1\a_3\\
1&\a_5&\a_4&\a_4\a_5&1&\a_5&\a_4&\a_4\a_5\\
\end{pmatrix}
\]

Thus, the set $Y=\V_\F(\Ss)$ is the union
\begin{multline*}
Y=\{(t_1\a_1,t_1t_2\a_2\a_4\a_5,t_1t_2,t_2)\}\cup\{(t_1\a_1,t_1t_2\a_2\a_4,t_1t_2\a_3,t_2\a_5)\}\cup\\
\{(t_1\a_1,t_1t_2\a_2\a_5,t_1t_2\a_3,t_2\a_4)\}\cup\{(t_1\a_1,t_1t_2\a_2,t_1t_2\a_3,t_2\a_4\a_5)\}\cup\\
\{(t_1,t_1t_2\a_2\a_4\a_5,t_1t_2\a_1\a_3,t_2)\}\cup\{(t_1,t_1t_2\a_2\a_4,t_1t_2\a_1\a_3,t_2\a_5)\}\cup\\
\{(t_1,t_1t_2\a_2\a_5,t_1t_2\a_1\a_3,t_2\a_4)\}\cup\{(t_1,t_1t_2\a_2\a_5,t_1t_2\a_1\a_3,t_2\a_4\a_5)\},
\end{multline*}
where $t_i\in\F$.
\end{example}

The information of the author:

Artem N. Shevlyakov

Omsk Branch of Institute of Mathematics, Siberian Branch of the Russian Academy of Sciences

644099 Russia, Omsk, Pevtsova st. 13

Phone: +7-3812-23-25-51.

e-mail: \texttt{a\_shevl@mail.ru}
\end{document}